\documentclass{amsart}[10pt]
\usepackage[matrix,arrow,curve]{xy}
\usepackage{amssymb, amsmath, euscript}
\newtheorem{theorem}{Theorem}
\newtheorem*{theorem*}{theorem}

\newtheorem{lemma}{Lemma}
\newtheorem{definition}{Definition}
\newtheorem{corollary}{Corollary}

\renewcommand{\to}[1][]{\xrightarrow{#1}}
\begin{document}
\author{Alexey Petukhov}
\title[Support varieties of $(\mathfrak g, \mathfrak k)$-modules of finite type]{Support varieties of $(\mathfrak g, \mathfrak k)$-modules of finite type}
\maketitle
\section{Brief statement of results}
Let $\mathfrak g$ be a semisimple Lie algebra over an algebraically closed field $\mathbb F$ of characteristic 0 and $\mathfrak k\subset\mathfrak g$ be a reductive in $\mathfrak g$ subalgebra.
\begin{definition}\upshape A {\it $(\mathfrak g, \mathfrak k)$-module} is a $\mathfrak g$-module which after restriction to $\mathfrak k$ becomes a direct sum of finite-dimensional $\mathfrak k$-modules.\end{definition}
\begin{definition}\upshape A $(\mathfrak g, \mathfrak k)$-module is of {\it finite type} if it is a $\mathfrak k$-module of {\it finite type}, i.e. has finite-dimensional $\mathfrak k$-isotypic components.\end{definition}
Let Z$(\mathfrak g)$ be the center of the universal enveloping algebra U$(\mathfrak g)$. Let $\chi:$Z$(\mathfrak g)\to \mathbb F$ be an algebra homomorphism.
\begin{definition}\upshape We say that a $\mathfrak g$-module $M$ {\it affords a central character} if for some homomorphism $\chi:$Z$(\mathfrak g)\to\mathbb F$ we have $zm=\chi(z)m$ for all $z\in\mathrm Z(\mathfrak g)$ and $m\in M$.\end{definition}
Any simple $\mathfrak g$-module $M$ affords a central character~\cite{Dix}. Let $X$ be the variety of all Borel subalgebras of $\mathfrak g$. The category of $\mathfrak g$-modules which affords a central character $\chi$  is equivalent to the category of sheaves of $\EuScript O_X$-quasicoherent modules over the sheaf of twisted differential operators $\EuScript D^\lambda(X)$ for a suitable twist $\lambda\in$H$^1(X, \Omega^{1,~cl}_X)$~\cite{BeBe}, where $\Omega^{1,~cl}_X$ is the sheaf of closed holomorphic 1-forms on $X$. In this category there is a distinguished full subcategory of holonomic sheaves of modules. Informally, holonomic sheaves of modules are $\EuScript D^\lambda(X)$-modules of minimal growth. The simple holonomic modules $M$ are in one-to-one correspondence with the pairs $(L, S)$, where $L$ is an irreducible closed subvariety of $X$ and $S$ is a sheaf of $\EuScript D^\lambda (L')$-modules which is $\EuScript O(L')$-coherent after restriction to a suitable open subset $L'\subset L$. Moreover, a coherent holonomic module $S$ is locally free on $L'$ and one could think about it as a vector bundle $S_B$ over $L'$ with a flat connection. Note that flat local sections of this bundle are not necessarily algebraic.
\begin{theorem}\upshape Let $M$ be a finitely generated $(\mathfrak g, \mathfrak k)$-module of finite type which affords a central character. Then $\mathrm{Ind}M$ is a holonomic $\EuScript D^\lambda (X)$-module.\end{theorem}
We also prove the following theorem(the necessary definitions see in the following section).
\begin{theorem}\upshape\label{G}Let $\EuScript O\subset\mathfrak g^*$ be a nilpotent coadjoint $G$-orbit, $\mathfrak k^\bot$ be the annihilator of $\mathfrak k$ in $\mathfrak g^*$, and $\mathrm N_{\mathfrak k}\mathfrak g^*$ be the $\mathfrak k$-null-cone in $\mathfrak g^*$. Then the irreducible components of $\EuScript O\cap \mathfrak k^\bot\cap\mathrm N_{\mathfrak k}\mathfrak g^*$ are isotropic subvarieties of $\EuScript O$.\end{theorem}
Let $\mathrm{V}^._{\mathfrak g, \mathfrak k}$ be the set of all irreducible components of possible intersections of $\mathrm N_K\mathfrak k^\bot$ with the $G$-orbits in $\mathrm N_G\mathfrak g^*$. This finite set of subvarieties of $\mathfrak g^*$ determines a finite set $\EuScript V^._{\mathfrak g, \mathfrak k}$ of subvarieties of T$^*X$ and a finite set $\mathrm L^._{\mathfrak g, \mathfrak k}$ of subvarieties of $X$ (see Definition~\ref{kus} below).
\begin{theorem}\upshape Let $M$ be a finitely generated $(\mathfrak g, \mathfrak k)$-module of finite type which affords a central character and $(L, S)$ be the corresponding pair consisting of a variety and a coherent sheaf as before. Then $L$ is an element of $\mathrm L^._{\mathfrak g, \mathfrak k}$.\end{theorem}
\section{Preliminaries}
We work in the category of algebraic varieties over $\mathbb F$. By T$^*X$ we denote the total space of the cotangent bundle of a smooth variety $X$ and by T$^*_xX$ the cotangent space to $X$ at a point $x$. By N$_{Y/X}^*\subset$T$^*X|_Y$ we denote the conormal bundle to a smooth subvariety $Y\subset X$.
\subsection{D-modules versus $\mathfrak g$-modules.}
Let $\mathfrak h$ be a Cartan subalgebra of $\mathfrak g$ and $\Sigma\subset\mathfrak h$ be the root system of $\mathfrak g$ and $\Sigma_+$ be a set of positive roots. Denote by $\hat{\mathfrak h^*}$ the set of weights $\lambda$ such that $\alpha^\vee(\lambda)$ is not a strictly positive integer for any positive root $\alpha^\vee$ of the dual root system $\Sigma^\vee\subset\mathfrak h^*$.

For a fixed $\lambda$ we denote by $\EuScript D^{\lambda}(X)$ the sheaf of twisted differential operators on $X$ and by D$^{\lambda}(X)$ its space of global sections. The algebras D$^{\lambda} (X)$ and D$^{\mu} (X)$ are naturally identified if $\lambda$ and $\mu$ lie in one shifted orbit of Weyl group~\cite{HDWJ}. Moreover, any such orbit intersects $\hat{\mathfrak h^*}$~\cite{HDWJ}. If $\lambda\in\hat{\mathfrak h^*}$ the surjective homomorphism \begin{center}$\tau:\mathrm U(\mathfrak g)\to\mathrm D^{\lambda}(X)$\end{center} identifies the category of quasicoherent D$^{\lambda} (X)$-modules and the category of $\mathfrak g$-modules affording the central character $\chi=\chi_\lambda$. A. Beilinson and J. Bernstein have proved that both above categories are equivalent to the category of $\EuScript D^{\lambda}(X)$-modules:\\
$$\begin{tabular}{c|c}Res: $\EuScript D^\lambda(X)-$mod$\longrightarrow$ $\mathfrak g$-mod$^\chi$&Ind: $\EuScript D^\lambda(X)$-mod$\longleftarrow\mathfrak g$-mod$^\chi$\\$\EuScript F\to \Gamma(X, \EuScript F)$&\hspace{26pt}$M\otimes_{(1\otimes\tau)\mathrm U(g)}\EuScript D(X)\longleftarrow M$\hspace{10pt}.\\\end{tabular}$$

In general Res and Ind identify the category $\mathfrak g$-mod$^\chi$ with a certain quotient of the category $\EuScript D^\lambda(X)$-mod. For more detailed exposition of the topic see for example~\cite{HDWJ}.
\subsection{Three faces of the support variety.}
The algebra U$(\mathfrak g)$ has a natural filtration such that gr U$(\mathfrak g)$=S$(\mathfrak g)$. The filtration on U$(\mathfrak g)$ induces a filtration on any finitely generated $\mathfrak g$-module $M$. We denote the associated graded S$(\mathfrak g)$-module by gr$M$. We denote the support of gr$M$ inside $\mathfrak g^*=$Spec~S$(\mathfrak g)$ by V$(M)$.

A similar argument constructs for a $\EuScript D^\lambda(X)$-module $\EuScript F$ a positive cycle $\EuScript V(\EuScript F)$, i.e. a formal linear combination of irreducible subvarieties of T$^*X$ with positive integer coefficients.
\begin{definition}\upshape The {\it singular support} $\EuScript V(M)$ of a simple $\mathfrak g$-module $M$ is the cycle $\EuScript V(\mathrm{Ind}M)$ in $\mathrm T^*X$.\end{definition}
\begin{definition}\upshape The {\it support variety} $\mathrm L(M)$ of a simple $\mathfrak g$-module $M$ is the projection of $\EuScript V(M)$ to $X$.\end{definition}
Let $\EuScript X$ be a $G$-variety for a reductive group $G$ with Lie algebra $\mathfrak g$. The map $\phi: $T$^*\EuScript X\times\mathfrak g\to\mathbb F ( \{(l,x),g\}\to l(gx), x\in X, l\in$T$^*_x\EuScript X, g\in\mathfrak g)$ determines a map $\phi_\EuScript X:$T$^*\EuScript X\to\mathfrak g^*$ called the {\it moment map}. D.Barlet and M.Kashiwara~\cite{BaKs}, have proved that V$(M)=\phi_X(\EuScript V($Ind$M))$. Therefore we have a diagram
\begin{center}\hspace{40pt}$\xymatrix{&\EuScript V(M)\subset\mathrm T^*X\ar[ld]^{\phi_X}\ar[rd]^{\mathrm{pr}}& &\\\mathrm V(M)\subset\mathfrak g^*&&\mathrm L(M)\subset X\hspace{39pt}.}$\end{center}
\begin{lemma}[\cite{VP}]\upshape  Let $K$ be a reductive algebraic group with a Lie algebra $\mathfrak k$ and let $X$ be an affine $K$-variety. Then $\mathbb F[X]$ is a $\mathfrak k$-module of finite type if and only if $X$ contains finitely many closed $K$-orbits. In this case any irreducible component of $X$ contains precisely one closed $K$-orbit.\end{lemma}
\begin{lemma}\upshape A finitely generated $(\mathfrak g, \mathfrak k)$-module $M$ is of finite type if and only if its associated variety $\mathrm V(M)$ has finitely many closed $\mathfrak k$-orbits. In this case the set of closed orbits consists just of the zero orbit.\end{lemma}
\begin{proof} Let J$_M$ be the annihilator of V$(M)$ in S$(\mathfrak g)$. Consider the S$(\mathfrak g)$-modules\begin{center}  J$_M^{-i}\{0\}:=\{m\in$gr$M\mid j_1...j_im=0$ for all $j_1,...,j_i\in$J$_M~\}$.\end{center} One can easily see that these modules form an ascending filtration of gr$M$ such that $\cup_{i=1}^\infty$J$_M^{-i}\{0\}=$gr$M$. Since S$(\mathfrak g)$ is a Noetherian ring, the filtration stabilizes, i.e. J$_M^{-i}\{0\}=$gr$M$ for some $i$. By $\overline{\mathrm{gr}}M$ we denote the corresponding graded object. By definition, $\overline{\mathrm{gr}}M$ is an S$(\mathfrak g)/$J$_M$-module. Suppose that $f\overline{\mathrm{gr}}M=0$ for some $f\in$S$(\mathfrak g)$. Then $f^i$gr$M=0$ and hence $f\in$J$_M$. This proves that the annihilator of $\mathrm{\overline{gr}}M$ in $\mathrm S(\mathfrak g)/\mathrm J_M$ equals zero.

Suppose V$(M)$ has a unique closed $\mathfrak k$-orbit.  Let $M_0$ be a $\mathfrak k$-stable space of generators of $\mathrm{\overline{gr}}M$. Then there is a surjective homomorphism $M_0\otimes_\mathbb F$(S$(\mathfrak g)/$J$_M)\to\overline{\mathrm{gr}}M$. Since V$(M)$ has finitely many closed $\mathfrak k$-orbits, $M_0\otimes_\mathbb F$(S$(\mathfrak g)/$J$_M$) is a $\mathfrak k$-module of finite type. Therefore $\overline{\mathrm gr}M$ is of finite type, which implies that $M$ is a  $(\mathfrak g, \mathfrak k)$-module of finite type.

Assume now that $M$ is of finite $\mathfrak k$-type. Set\begin{center}Rad $M=\{m\in\overline{\mathrm{gr}}M|$ there exists $f\in$S$(\mathfrak g)/$J$_M$ such that $fm=0$ and $f\ne 0\}$.\end{center} Then Rad $M$ is a proper $\mathfrak k$-stable submodule of $\overline{\mathrm{gr}}M$. Therefore there exists a finite-dimensional $\mathfrak k$-subspace $M_0\subset\overline{\mathrm{gr}}M$ such that $M_0\cap$Rad$M$=0. The homomorphism $M_0\otimes_\mathbb F$(S$(\mathfrak g)/$J$_M)\to \overline{\mathrm{gr}}M$ induces an injective homomorphism S$(\mathfrak g)/$J$_M\to M_0^*\otimes_\mathbb F\overline{\mathrm{gr}}M$. Therefore S$(\mathfrak g)/$J$_M$ is of finite $\mathfrak k$-type and V($M$) has only finitely many closed $\mathfrak k$-orbits. As V$(M)$ is $\mathbb F^*$-stable, any irreducible component of it contains point 0 and this point is a closed $\mathfrak k$-orbit.\end{proof}
\begin{lemma}[S. Fernando~\cite{F}]\upshape If $M$ is a finitely generated $(\mathfrak g, \mathfrak k)$-module then $$\mathrm V(M)\subset\mathfrak k^\bot:=\{x\in\mathfrak g^*:\forall k\in\mathfrak k~x(k)=0\}~~.$$\end{lemma}
\subsection{Hilbert-Mumford criterion.}Let $K$ be a reductive group, $X$ be an affine $K$-variety, $V$ be a $K$-module.
\begin{theorem}[Hilbert-Mumford criterion]\upshape The closure of any orbit $\overline{Kx}\subset X$ contains a unique closed orbit $K\overline x\subset X$. There exists a homomorphism $\mu:\mathbb F^*\to K$ such that $\lim\limits_{t\to 0}\mu(t)x=\bar x\in K\overline x$.\end{theorem}
The {\it null-cone} N$_{\mathfrak k}V:=\{x\in V: 0\subset \overline{Kx}\}$ is a closed algebraic subvariety of $V$~\cite{VP}.
\begin{theorem}\upshape Let $x\in V$ be a point. Then $0\in \overline{Kx}$ if and only if there exists a nonzero rational semisimple element $h\in\mathfrak k$ such that $x\in V^{>0}_h$, where $V^{> 0}_h$ is the direct sum of $h$-eigenspaces in $V$ with positive eigenvalues.\end{theorem}
\begin{corollary}[\cite{VP}]\upshape \label{HM}There exists a finite set $H$ of rational semisimple elements of $\mathfrak k$ such that $\mathrm N_KV:=\cup_{h\in H}KV_h^{>0}$, where $KV_h^{>0}:=\{v\in V\mid v=kv_h$ for some $k\in K$ and $v_h\in V_h^{>0}$\}.\end{corollary}
\subsection{Gabber's theorem.}
Let $G$ be the adjoint group of $\mathfrak g$ and $M$ be a finitely generated $\mathfrak g$-module.
\begin{definition}\upshape Suppose $\EuScript X$ is a smooth $G$-variety with a closed $G$-invariant nondegenerate 2-form $\omega$. Such a pair $(\EuScript X,\omega)$ is called a {\it symplectic $G$-variety}. \end{definition}
\begin{definition}\upshape Let $(\EuScript X,\omega)$ be a symplectic $G$-variety. We call a subvariety $Y\subset \EuScript X$\\a){\it isotropic} if $\omega|_{\mathrm T_yY}=0$ for a generic point $y\in Y$;\\b){\it coisotropic} if $\omega|_{(\mathrm T_yY)^\bot}=0$ for a generic point $y\in Y$;\\c){\it Lagrangian} if $\mathrm T_yY=(\mathrm T_yY)^\bot$ for a generic point $y\in Y$ or equivalently if it is both isotropic and coisotropic.\end{definition}
\begin{theorem}[O. Gabber ~\cite{Gab}]\label{Gab}\upshape The variety $\mathrm V(M)\subset\mathrm N_G\mathfrak g^*$ is a coisotropic subvariety of $\mathfrak g^*$ with respect to the Kirillov symplectic structure. The variety $\EuScript V(M)$ is a coisotropic subvariety of $\mathrm T^*X$ with respect to the natural symplectic structure.\end{theorem}
\begin{definition}\upshape A finitely generated $(\mathfrak g, \mathfrak k)$-module $M$ which affords a central character is called {\it holonomic} if $\tilde V$ is a Lagrangian subvariety of $G\tilde V$ for any irreducible component $\tilde V$ of $\mathrm V(M)$.\end{definition}
\section{Proofs}
Let $\mathfrak g$ be a semisimple Lie algebra and $\mathfrak k\subset\mathfrak g$ be a reductive in $\mathfrak g$ subalgebra. Let $G$ be the adjoint group of $\mathfrak g$ and $K\subset G$ be a connected reductive subgroup such that lie~$K=\mathfrak k$.
\setcounter{theorem}{1}\begin{theorem}\upshape Let $\EuScript O\subset\mathfrak g^*$ be a nilpotent coadjoint orbit, $\mathfrak k^\bot$ be the annihilator of $\mathfrak k$ in $\mathfrak g^*$, and $\mathrm N_{\mathfrak k}\mathfrak g^*$ be the $\mathfrak k$-null-cone in $\mathfrak g^*$. Then the irreducible components of $\EuScript O\cap \mathfrak k^\bot\cap\mathrm N_{\mathfrak k}\mathfrak g^*$ are isotropic subvarieties of $\EuScript O$.\end{theorem}
\begin{proof} As $\mathfrak g$ is semisimple, we can freely identify $\mathfrak g$ with $\mathfrak g^*$. Let $h$ be a nonzero rational semisimple element of $\mathfrak k$. By definition $\mathfrak g^{\ge 0}_h$ is the direct sum of all $h$-eigenspaces in $\mathfrak g$ with nonnegative eigenvalues. Let $G_h^{\ge 0}\subset G$ be the parabolic subgroup with Lie algebra $\mathfrak g_h^{\ge 0}$, and let $S_G:=G/G_h^{\ge 0}$ be a quotient. In the same way we define $\mathfrak k_h^{\ge 0}, K_h^{\ge 0}, S_K$. Let $\mathfrak n_h$ be the nilpotent radical of $\mathfrak g_h^{\ge 0}$. Obviously $Ke\subset S_G$ is isomorphic to $S_K$. By definition,\\$\bullet~G\mathfrak n_h:=\{x\in\mathfrak g\mid x=gn$ for some $n\in\mathfrak n_h, g\in G$\},\\$\bullet~K\mathfrak n_h:=\{x\in\mathfrak g\mid x=kn$ for some $n\in\mathfrak n_h, k\in K$\},\\$\bullet~K(\mathfrak n_h\cap\mathfrak k^\bot):=\{x\in\mathfrak g\mid x=kn$ for some $n\in\mathfrak n_h\cap\mathfrak k^\bot, k\in K$\}.\\

Let $\phi:$~T$^*S_G\to\mathfrak g^*$ be the moment map. It is a straightforward observation that $G\mathfrak n_h$ coincides with $\phi($T$^*S_G)$, $K\mathfrak n_h$ coincides with $\phi($T$^*S_G|_{S_K})$, $K(\mathfrak n_h\cap\mathfrak k^\bot)$ coincides with $\phi($N$_{S_K/S_G}^*)$.
$$\xymatrix{\mathrm T^*S_G\ar[d]^\phi&\mathrm T^*S_G|_{S_K}\ar[d]^\phi\ar[l]^{\text{inclusion}}&\mathrm N^*_{S_K/S_G}\ar[d]^\phi\ar[l]^{\text{inclusion}}\\G\mathfrak n_h&K\mathfrak n_h\ar[l]^{\text{inclusion}}&K(\mathfrak n_h\cap\mathfrak k^\bot)\ar[l]^{\text{inclusion}}}$$
As N$^*_{S_K/S_G}$ is an isotropic subvariety of T$^*S_G$, the variety $\phi($N$^*_{S_K/S_G}$) is isotropic in $G(\mathfrak n_h\cap\mathfrak k^\bot)$ and any subvariety $\tilde V$ of $K(\mathfrak n_h\cap\mathfrak k^\bot)$ is isotropic in $G\tilde V$. Therefore by Corollary~\ref{HM} any subvariety $\tilde V$ of $\mathfrak k^\bot\cap\mathrm N_{\mathfrak k}\mathfrak g$ is isotropic in $G\tilde V$.\end{proof}

\begin{definition}\label{kus}\upshape~\\$\bullet$~~~~Let $\mathrm{V}^._{\mathfrak g, \mathfrak k}$ be the set of all irreducible components of all possible intersections of $\mathrm N_K\mathfrak k^\bot$ with a $G$-orbit in $\mathrm N_G\mathfrak g^*$.\\$\bullet$ Let $\EuScript V^._{\mathfrak g, \mathfrak k}$ be the set of all irreducible components of the preimages of $\mathrm V^._{\mathfrak g, \mathfrak k}$ under the moment map $\mathrm T^*X\to\mathfrak g^*$.\\$\bullet$ Let $\mathrm L^._{\mathfrak g, \mathfrak k}$ be the set of all images of elements of $\EuScript V^._{\mathfrak g, \mathfrak k}$ in $X$.\end{definition}
\setcounter{theorem}{2}\begin{theorem}\upshape Let $M$ be a finitely generated $(\mathfrak g, \mathfrak k)$-module of finite type which affords a central character. The irreducible components of $\mathrm V(M)$ are elements of $\mathrm V^._{\mathfrak g, \mathfrak k}$, the irreducible components of $\EuScript V(M)$ are elements of $\EuScript V^._{\mathfrak g, \mathfrak k}$, the irreducible components of $\mathrm L(M)$ are elements of $\mathrm L^._{\mathfrak g, \mathfrak k}$.\end{theorem}
\begin{proof}Let $\tilde V$ be an irreducible component of V$(M)$ and $\EuScript O$ be the closure of $G\tilde V$. By Theorem~\ref{Gab} the variety $\tilde V$ is coisotropic. As $\tilde V\subset$N$_K\mathfrak k^\bot\cap \EuScript O$, $\tilde V$ is isotropic, and therefore $\tilde V$ is Lagrangian and is an irreducible component of $\EuScript O\cap$N$_K\mathfrak g^*\cap\mathfrak k^\bot$.\end{proof}
\begin{proof}[Proof of Theorem 1]As any irreducible component $\tilde{\EuScript V}$ of $\EuScript V(M)$ is Lagrangian in T$^*X$, the module Ind$M$ is holonomic.\end{proof}
\section{Acknowledgements}I thank my scientific advisor Ivan Penkov for his attention to my work and the great help with the text-editing.


\begin{thebibliography}{99}
\bibitem{BaKs} Daniel Barlet, Masaki Kashiwara, Duality of D-modules on Flag manifolds, IMRN(2000), p. 1243--1257
\bibitem{BeBe} Alexandr Beilinson, Joseph Bernstein, Localisation de $\mathfrak g$-modules (French),  C. R. Acad. Sci. Paris Ser. I Math. {\bf 292}(1981), p. 1--18
\bibitem{Be} Alexandr Beilinson, Localization of representations of reductive Lie algebras,  Proc. of the ICM 1983, p. 699--710
\bibitem{BoBry} Walter Borho, Jean-Luc Brylinski, Differential operators on homogeneous spaces. III, Invent. Math. {\bfseries 80}(1985), p. 1--68
\bibitem{Dix} Jacques Dixmier, Alg\' ebres Enveloppantes, Gauthier-Villars, Paris, 1974
\bibitem{F} Suren L. Fernando, Lie algebra modules with finite dimensional weight spaces. I", Trans. Amer. Math. Soc. {\bfseries 322} (1990), p. 757–-781
\bibitem{Gab} Ofer Gabber,  The integrability of the characteristic variety,  Amer. J. Math.  {\bfseries 103} no. 3(1981), p. 445--468
\bibitem{HDWJ}  Henryk Hecht, Dragan Milicic, Wilfried Schmidt, Localization and standard modules for real semisimple Lie groups I: The duality theorem, Inv.
\bibitem{PSZ} Ivan Penkov, Vera Serganova, Gregg Zuckerman, On the existence of ($\mathfrak g, \mathfrak k)$-modules of finite type, Duke Math. J. {\bfseries 125} (2004), p. 329--349
\bibitem{VP}  Ernest B. Vinberg, Vladimir L. Popov, Invariant theory (Russian) Algebraic geometry, 4 (Russian), Itogi Nauki i Tekhniki, Akad. Nauk SSSR,
Vsesoyuz. Inst. Nauchn. i Tekhn. Inform., Moscow, 1989, p. 137--314
\end{thebibliography}
\end{document}